\documentclass[10pt, a4paper]{article}

\usepackage{a4wide}

\usepackage[utf8]{inputenc}
\usepackage[english]{babel}
\usepackage{amsthm}
\usepackage{amsmath}
\usepackage{amssymb}
\usepackage{url}
\usepackage[onehalfspacing]{setspace}

\usepackage{verbatim}
\usepackage{alltt}
\usepackage{fancyvrb}
\usepackage{float}
\usepackage{etoolbox}
\usepackage{enumerate}
\theoremstyle{plain}
\newtheorem{thm}{Theorem}

\newtheorem{innercustomthm}{Theorem}
\newenvironment{customthm}[1]
  {\renewcommand\theinnercustomthm{#1}\innercustomthm}
  {\endinnercustomthm}

\theoremstyle{definition}

\theoremstyle{remark}
\newtheorem{remark}[thm]{Remark}

\newtheorem{example}{Example}[section]
\newcommand{\F}{\mathbb{F}}
\newcommand{\E}{\mathbf{E}}

\begin{document}
\date{}
\title{The measures of pseudorandomness and the NIST tests}
\author{L\'aszl\'o M\'erai \\	
	 \multicolumn{1}{p{.7\textwidth}}{\small \centering\emph{\small Johann Radon Institute for Computational and Applied Mathematics, Austrian Academy of Sciences, Altenbergerstr.\ 69, 4040 Linz, Austria}\\ 
	e-mail: \texttt{merai@cs.elte.hu}}\\ \\
	Jo\"el Rivat \\
	\multicolumn{1}{p{.7\textwidth}}{\small \centering\emph{\small Institut de Math\'ematiques de Marseille UMR 7373, Universit\'e d'Aix-Marseille, 163 avenue de Luminy, 13288 Marseille Cedex 9, France} \\ 
	e-mail: \texttt{joel.rivat@univ-amu.fr}}\\ \\
	and\\
	Andr\'as S\'ark\"ozy \\
	\multicolumn{1}{p{.7\textwidth}}{\small \centering\emph{\small E\"otv\"os Lor\'and University, Department of Algebra and Number Theory,  P\'azm\'any P\'eter s\'et\'any 1/c, H-1117 Budapest, Hungary}\\  e-mail: \texttt{sarkozy@cs.elte.hu}}
	}

\maketitle

\begin{abstract}
  A few years ago new quantitative measures of pseudorandomness of
  binary sequences have been introduced. Since that these measures
  have been studied in many papers and many constructions have been
  given along these lines. In this paper the connection between the
  new measures and the NIST tests is analyzed. It is shown that finite 
  binary sequences possessing strong pseudorandom properties in terms 
  of these new measures usually also pass or nearly pass most of the 
  NIST tests. 

 {
 \let\thefootnote\relax\footnote{The final publication is available at Springer via \url{https://doi.org/10.1007/978-3-319-76620-1_12} \\
 Research partially supported by the Hungarian National Foundation for Scientific Research, grants K119528 and NK104183, by the Austrian Science Fund FWF Project F5511-N26 which is part of the Special Research Program "Quasi-Monte Carlo Methods: Theory and Applications" and by the ANR Project MuDeRa.}
 }

\end{abstract}



\section{Introduction}

The National Institute of Standards and Technology (=NIST) of the US issued 
the document \cite{nist} which we refer to as ``NIST tests''. We quote the 
introduction of this documentum: ``The need for random and pseudorandom
numbers arises in many cryptographic applications. For example, common 
cryptosystem...'' [e.g., the Vernam cipher] ``...employs keys that must be
generated in a random fashion... This document discusses the randomness testing 
of random numbers and pseudorandom number generators that may be used for many
purposes including cryptographic, modeling and simulation applications. The
focus of this document is on those applications where randomness is required
for cryptographic purposes. A set of statistical tests for randomness is
described in this document.'' The NIST tests
is a package consisting of 15 tests, and in each of these 15 cases one has to
compute the value of a certain statistics composed from the elements of the
given sequence. Then we have to check whether this value is close enough to
the expected value of this statistics for a random binary sequence.
If, say, we want to check the quality of a PRBG 
(=pseudorandom bit generator; an algorithm generating a long bit sequence from
a short random one called ``seed''), then this can be done by testing several
bit sequences generated from random seeds by the PRBG; if these sequences 
pass the NIST tests then the PRBG is suitable for further consideration.
As the NIST tests writes: 
``These tests may be useful as a first step in determining whether or not a
generator is suitable for a particular cryptographic application. However, 
no set of statistical tests can absolutely certify a generator as appropriate
for usage in a particular application, i.e., statistical testing cannot serve 
as a substitute for cryptanalysis.'' The weak point of this 
``first step'' by using the NIST tests is that they are of a posteriori type, 
i.e., we do not have any a priori control of the pseudorandom quality 
of the output sequences of the PRBG so that we do not know 
anything about the output sequences not tested by the NIST tests. 

Thus one might like to replace this a posteriori type testing based on the
NIST tests with a method for a priori testing 
(called ``theoretical testing'' by Knuth) of all the output sequences of 
the PRBG (which seems to be a too optimistic goal) or at least to combine 
and complete the NIST tests by a method of this type
(this is a more realistic goal). In 1997 Mauduit and the third author
\cite{MS1997} made a significant step in this direction: they
introduced certain measures of pseudorandomness, and they presented an
example for binary sequence which possess strong pseudorandom
properties in terms of these measures. Since that more than 150 papers
have been written in which these measures are studied, further
measures are introduced, or further ``good'' constructions are
presented; an excellent survey of these papers is given by Gyarmati
\cite{Gy2013}. It is a natural question to ask: how is this direction
related to the NIST tests? Can one, indeed, complete
the a posteriori testing by using these new results? A partial answer
was given by the second and third author in \cite{RS2006}: they
studied the connection of 3 NIST tests and a further often used test
with the measures of pseudorandomness
introduced in \cite{MS1997}, and they showed that the values of the
statistics to be computed in each of these tests can be estimated well
by using the measures of pseudorandomness mentioned above, moreover,
they also presented numerical calculations to show that a ``random''
sequence selected from a family of binary sequences constructed by
using the Legendre symbol \cite{GMS2004,MS1997} passes all the NIST
tests (of 2005). In this paper our goal is to continue that work in
the following direction: we will study the connection between 3 further NIST
tests and our measures of pseudorandomness. (The 8 remaining NIST
tests are too complicated to study their theoretical connection with
the measures of pseudorandomness.) Moreover, we will present further
numerical calculations to show that a ``random'' sequence selected
from two other families constructed by using other principles also
passes or ``almost passes'' all the NIST tests.

\section{The measures of pseudorandomness}\label{sec:measures}

First we recall a few definitions and facts from \cite{MS1997} and other related papers that we will need in this paper.

Consider a finite binary sequence
\begin{equation}\label{eq:E}
 E_N=(e_1,\dots,e_N)\in\{-1,+1\}^N.
\end{equation}
(Note that in the analysis in some of the NIST  tests the bit sequences are also transformed into sequences consisting of -1 and +1.) Then the \emph{well-distribution measure} of $E_N$ is defined as
\begin{equation}\label{eq:W}
W(E_N)=
\max _{a,b,t} \left| \sum_{j=0}^{t-1} e_{a+jb} \right|,
\end{equation}
where the maximum is taken over all $a,b,t \in \mathbb{N}$ such that $1 \leq a \leq a+(t- 1)b\leq N$, while the \textit{correlation measure of order $k$} of $E_N$ is defined as
\begin{equation}\label{eq:C}
C_{k}(E_N)=\max_{M,D}\left| \sum_{n=1}^M e_{n+d_1}e_{n+d_2}\dots e_{n+d_{k}} \right|,
\end{equation}
where the maximum is taken over all $D=(d_1\dots, d_k)$ and $M$ such that $0\leq d_1<\dots< d_{k} \leq N-M$. Then the sequence is considered as a ``good'' pseudorandom sequence if both these measures $W(E_N)$ and $C_k(E_N)$ (at least for ``small'' $k$) are ``small'' in terms of $N$ (in particular, both are $o(N)$ as $N\rightarrow \infty$). Indeed, it is shown in \cite{CMS2002} that for a ``truly random'' $E_N\in\{-1,+1\}^N$ both $W(E_N)$ and, for fixed $k$, $C_k(E_N)$ are of order of magnitude $N^{1/2}$ with probability ``near 1'' (see also \cite{AKMMR2007} and   \cite{KMMR2003}). Thus for ``really good'' pseudorandom sequences we expect the measures \eqref{eq:W} and \eqref{eq:C} to be not much greater than $N^{1/2}$.
In \cite{MS1997} a combination of the well-distribution and correlation measures was also introduced: the \emph{combined pseudorandom measure of order $k$} of the sequence $E_N$ in \eqref{eq:E} is defined as
\[
 Q_k(E_N)=\max_{a,b,t,D}\left| \sum_{j=0}^t e_{a+jb+d_1}e_{a+jb+d_2}\dots e_{a+jb+d_{k}} \right|
\]
where the maximum is taken over all $a,b,t\in\mathbb{N}$ and $k$-tuples $D=(d_1,d_2,\dots,d_k)$ of non-negative integers $d_1<d_2<\dots<d_k$ such that all the subscripts $a+jb+d_l$ belong to $\{1,2,\dots,N\}$. (Clearly, we have $W(E_N)=Q_1(E_N)$ and $C_k(E_N)\leq Q_k(E_N)$ for $k\geq 2$.)
We will also need the definition of normality measures also introduced in \cite{MS1997}. The \emph{normality measure of order $k$} of the sequence $E_N$ of form  \eqref{eq:E} is defined as
\begin{multline}\label{eq:N}
 N_k(E_N)=\max_{X\in\{-1,+1\}^k} \max_{0<M\leq N+1-k} \Big||\{n: 0\leq n\leq M, \ (e_{n+1},\dots, e_{n+k})=X  \}|-\frac{M}{2^k} \Big|.
\end{multline}
It was also shown in \cite{MS1997} (see Proposition 1 and its proof there) that for all $N$, $E_N$ and $k<N$ we have
\begin{equation}\label{eq:N-C}
 N_k(E_N)\leq \frac{1}{2^k}\sum_{t=1}^k \binom{k}{t} C_t(E_N)\leq \max_{1\leq t\leq k } C_t(E_N).
\end{equation}
Thus if $C_t(E_N)$ is small for all $t\leq k$, then $N_k(E_N)$ is also small.

\section{Three principles for constructing large families of binary sequences
  with strong pseudorandom properties}\label{sec:construction}

It is well known that the Legendre polynomial has many pseudorandom 
properties \cite{Cu,Da}. It was shown in \cite{MS1997} that the 
Legendre symbol also possesses strong pseudorandom properties in terms of 
the pseudorandom measures described in Section 2: 
if $p$ is an odd prime, we write $N=p-1$ and
\[
 E_N=(e_1,\dots, e_N) \quad \text{with } e_n=\left(\frac{n}{p} \right) \ \text{for } n=1,\dots, N,
\]
then we have
\[
 W(E_N)\ll N^{1/2} \log N \ \ \ \textup{ and } \ \ \ 
C_k(E_N)\ll kN^{1/2} \log N
\]
for all $k<N$ (where $\ll$ is Vinogradov's notation: $f(x)\ll g(x)$ means that $f(x)=O(g(x))$; in both cases the implicit constants can be
computed explicitly (and are relatively small constants). 

Goubin, Mauduit and S\'ark\"ozy \cite{GMS2004} studied the generalization of this construction with $f(n)$ in place of $n$ (where $f(x)\in\F_p[x]$). Their results can be combined in the following way:
\begin{customthm}{A}\label{thm:GMS2004}
If $p$ is a prime number, $f(x)\in\F_p[x]$ (\,$\F_p$ being the field of the modulo $p$ residue classes) has degree $k (>0)$, $f(x)$ has no multiple zero in $\overline{\F_p}$ (= the algebraic closure of $\F_p$), and the binary sequence $E_p=(e_1,\dots, e_p)$ is defined by
\begin{equation}\label{constr:legendre}
e_n=
\left\{
\begin{array}{cl}
 \left(\frac{f(n)}{p} \right) & \text{for } (f(n),p)=1\\
 +1 & \text{for } p\mid f(n),
\end{array}
\right.
\end{equation}
then we have
\[
 W(E_p)<10kp^{1/2}\log p.
\]
Moreover, assume that also $\ell\in\mathbb{N}$, and one of the following assumptions holds:
\begin{enumerate}[(i)]
 \item $\ell=2$; \label{ass:GMS-1}
 \item $\ell <p$, and 2 is a primitive root modulo $p$;\label{ass:GMS-2}
 \item $(4k)^\ell <p$.\label{ass:GMS-3}
\end{enumerate}
Then we also have
\[
 C_\ell (E_p)< 10 k \ell p^{1/2}\log p.
\]
\end{customthm}

The second principle is to utilize the fact that the multiplicative inverse
modulo $p$ is distributed in a random way in $(0,p)$. Denote the least
non-negative residue of $n$ modulo $p$ by $r_p(n)$, and if the prime $p$ is
fixed, then denote the multiplicative inverse of $a$ modulo $p$ by $a^{-1}$
(so that $a \cdot a^{-1}\equiv 1 \mod p$). The following theorem  
(here we present the result in a slightly simplified form) was
proved in \cite{MS2005}.
\begin{customthm}{B}\label{thm:inverse-W}
Assume that $p$ is an odd prime number, $f(x)\in\F_p[x]$ 
has degree $k$ with $0<k<p$
and no multiple zero in $\overline{\F_p}$. Define the binary sequence $E_p=(e_1,\dots, e_p)$ by
\begin{equation}\label{constr:inverse}
e_n=
\left\{
\begin{array}{cl}
 +1 & \text{if } (f(n),p)=1, \ r_p(f(n)^{-1})<\frac{p}{2}\\
 -1 & \text{if either } (f(n),p)=1, \ r_p(f(n)^{-1})>\frac{p}{2} \text{ or } p\mid f(n).
\end{array}
\right.
 \end{equation}
Then we have
\[
 W(E_p)\ll k p^{1/2}(\log p)^2.
\]
Moreover, if $\ell\in\mathbb N$, $2\le \ell\le \frac{p}{2^k}$ and 
$f(x)\in\mathbb F_p[x]$ is of the form $f(x)=(x+a_1)(x+a_2)\cdots (x+a_k)$ with
$a_1,\dots,a_k\in\mathbb{F}_p$ ($a_i\ne a_j$ for $i\ne j$)
then we also have
\begin{equation}\label{eq:inverse-C}
 C_\ell (E_p)\ll k \ell p^{1/2}(\log p)^{\ell +1}.
\end{equation}
\end{customthm}

For further related results see also \cite{CL2008,L2013,LG2014}. For example Liu \cite{L2013} gave another (and simpler) condition to control the correlation measure.

\begin{customthm}{C}\label{thm:inverse-Liu}
Assume that $p$ is an odd prime number, $f(x)\in\F_p[x]$ is a polynomial of degree $(0<)k(<p)$ such that 0 is its unique zero in $\F_p$. If the sequence $E_N$ is defined as in Theorem \ref{thm:inverse-W} and $\ell<p$, then \eqref{eq:inverse-C} also holds.
\end{customthm}

The third construction is based on elliptic curves. Let $p>3$ be a prime number and
let $\E$ be an elliptic curve over $\F_p$ defined by the Weierstrass equation
\[
y^2=x^3+Ax+B
\]
with coefficients $A, B\in \F_p$ and non-zero discriminant 
(see \cite{washington}).
The $\F_p$-rational points $\E(\F_p)$ of $\E$ form an Abelian group with the point in infinity $\mathcal{O}$ as the neutral element, where the group operation is denoted by $\oplus$. For a rational point $R\in \E(\F_p)$, a multiple of $R$ is defined by $nR=\bigoplus_{i=1}^nR$. Let $\F_p(\E)$ be the function field of $\E$ over $\F_p$ and as usual for $f\in\F_p(\E)$ let 
let $\deg f$ denote of the degree of $f$ in $\mathbb{F}_p(\mathbf{E})$, 
see \cite{washington}. For example, for the coordinate functions 
we have $\deg x=2$ and $\deg y=3$.

Let $G\in\E(\F_p)$ be of order $T$ and $f\in\F_p(\E)$. Define the binary sequence $E_T=(e_1,\dots, e_T)$ by
\begin{equation}\label{constr:ec}
e_n=
\left\{
\begin{array}{cl}
 \left(\frac{f(nG)}{p} \right) & \text{if } (f(nG),p)=1, \\
 +1 & \text{otherwise.} 
\end{array}
\right.
 \end{equation}

The first author studied the pseudorandomness of this sequence \cite{M2009}. His results can be combined in the following way:
\begin{customthm}{E}\label{thm:ec}
 Let $G\in\E(\F_p)$ be a generator of $\E(\F_p)$ of prime order $T$. Let $f\in\F_p(\E)$ which is not a perfect square in $\overline{\F_p}(\E)$ with degree $k=\deg f$. Then
 \[
  W(E_T)\leq 6 k p^{1/2}\log T.
 \]
Moreover, assume that also $\ell\in\mathbb{N}$, and one of the conditions \textit{(i),(ii),(iii)} of Theorem \ref{thm:GMS2004} holds with $p$ replaced by $T$. Then 
\[
 C_\ell (E_T)< 2 \ell k p^{1/2}\log T.
\]
\end{customthm}

\section{The ``frequency test within a block''}\label{sec:freqBlock}

First we will study the connection of this test (which appears as Section 2.2 in \cite{nist}) with the measures of pseudorandomness described in Section \ref{sec:measures}. We quote \cite{nist}: ``The focus of this test is to determine whether the frequency of ones in an $M$-bit block is approximately $M/2$, as would be expected under an assumption of randomness. For block size $M=1$, this test degenerates to test 1, the Frequency (Monobit) test'' (which was analyzed in \cite{RS2006}).

Let $E_N=(e_1,\dots, e_N)\in\{-1,+1\}^N$ be the sequence to be tested, $M$ the length of each block, and, as \cite{nist} writes, 
\begin{equation*}
  \parbox{\dimexpr\linewidth-4em}{
    \strut
    ``Partition the input sequence into $t=[\frac{N}{M}]$ non-overlapping blocks.'' 
    \strut
  }
\end{equation*}
(Here and later we adjust the notations of \cite{nist} to our notation.) The quotation continues:
\begin{equation*}
  \parbox{\dimexpr\linewidth-4em}{
    \strut
    ``Discard any unused bits. Determine the proportion $\pi_i$ of ones in each block of length $M$ for $1\leq i\leq t$'' 
    \strut
  }
\end{equation*}
Now ``Compute the $\chi^2$ statistic 
\begin{equation}\label{qu:block-3}
X_1=4M\sum_{i=1}^t\left(\pi_i-\frac{1}{2}\right)^2.'' 
\end{equation}
Then the sequence $E_N$ passes this test if the value of this statistic is small enough in the sense described in \cite{nist}; we skip the technical details.

In 2.2.7 \cite{nist} writes: ``The block size $M$ should be selected such that
\begin{equation}\label{qu:block-4}
 M\geq 20, \ M>N/100 \ \text{and } t<100."
\end{equation}


\begin{thm}
Using the notation above and assuming \eqref{qu:block-4}, for every $E_N\subset\{-1,+1\}^N$ we have
 \begin{equation}\label{eq:block-5}
  X_1\leq 2 \cdot 10^4 \frac{W(E_N)^2}{N}.
 \end{equation}
\end{thm}

\begin{proof}
Clearly we have
\begin{align*}
  \pi_i
  &=
  \frac{|\{e_j:\ (i-1)M<j\leq iM, e_j=+1\}|}{M}
  \\
  &=
  \frac{1}{M}\sum_{j=(i-1)M+1}^{iM} \frac{1}{2}(e_j+1)
  =\frac{1}{2M}\sum_{j=(i-1)M+1}^{iM} e_j+\frac{1}{2}
\end{align*}
whence, using the notation of Section \ref{sec:measures},
\[
 \left|\pi_i-\frac{1}{2} \right|=\frac{1}{2M}\left|\sum_{j=(i-1)M+1}^{iM} e_j \right|\leq \frac{1}{2M} W(E_N)
\]
for every $1\leq i \leq t$. Thus it follows from \eqref{qu:block-3} that
\begin{equation*}
X_1\leq 4M\cdot t\left( \frac{1}{2M} W(E_N) \right)^2 =\frac{2t}{M}W(E_N)^2.
\end{equation*}
By using \eqref{qu:block-4}, \eqref{eq:block-5} follows from this.
\end{proof}
%

In each of the constructions described in Section \ref{sec:construction} the upper bound in inequality \eqref{eq:block-5}  is less than a constant multiple of a fixed power of $\log N$, so that this upper bound falls just a little short of the desired $<c$ (with a small positive constant $c$). In many applications this can be interpreted as a strong tendency towards pseudorandomness which is sufficient for our purposes, while if we have to stick to the threshold bound belonging to the test, then this good upper bound points to the direction that choosing successive random sequences from our family studied we have a good chance to find soon a sequence which also satisfies the stronger inequality prescribed in the test.



\section{The ``test for the longest run of ones in a block''}\label{sec:longest_run}

This test appears in Section 2.4 of \cite{nist}. We quote \cite{nist}: ``The focus of the test is the longest run of ones within $M$ bit blocks. The purpose of this test is to determine whether the length of the longest run of ones within the tested sequence is consistent with the length of the largest run of ones that would be expected in a random sequence.''. The test to answer this question is carried out in \cite{nist} in the following way:

Assume the $N,M,t$ are positive integers with
\begin{equation}\label{eq:NMT}
 N=Mt,
\end{equation}
$N$ is the length of the sequence $E_N=(e_1,\dots, e_N)\in\{-1,+1\}^N$ to be tested (again  we switch from bit sequences to $\pm1$ sequences), $M$ is taken from a certain special sequence $8,128,10^4,\dots$ (see \cite{nist}), $E_N$ is split in $t$ blocks of length $M$, $t$ and thus also $N$ is large enough in terms of $M$ (in particular, for $M=8,128,10^4$ the number $N$ must be at least $128,272, 750000$, respectively) and $K$ ($=3,5,6,\dots$) is certain positive integer assigned to the given $M$ value. The set $\{0,1,\dots, M\}$ is split in $K+1$ disjoint parts $\mathcal{P}_0,\mathcal{P}_1,\dots, \mathcal{P}_K$ so that
\begin{equation}\label{eq:partition}
 \{0,1,\dots, M\}=\mathcal{P}_0 \cup \mathcal{P}_1 \cup \dots \cup \mathcal{P}_K, \quad \mathcal{P}_i\cap \mathcal{P}_j = \emptyset \quad \text{for } 0\leq i<j\leq K,
\end{equation}
e.g., for $M=10^4$, $K=6$ in \cite{nist} we have
\begin{align*}
  \{0,1,\dots,10^4\}
  =\{0,1,\dots, 10\}\cup\{11\}\cup\{12\}\cup\{13\}\cup\{14\}\cup\{15\}\cup\{16,17,\dots,10^4\}.  
\end{align*}
Then for $i=0,1,\dots, K$ we count how many of the $t$ blocks is such that the length of the longest run of $+1$'s in it belongs to the part $\mathcal{P}_i$ of $\{0,1,\dots,M\}$; let $\nu_i$ denote the number of blocks with this property. Let $\pi_i$ be the probability of the event that the length of the longest run of $+1$'s in a random sequence of $+1$ and $-1$ with length $M$ is $i$. The test statistic to be computed is a weighted square mean of the deviations of the $\nu_i$'s from their expected values $t\pi_i$:
\begin{equation}\label{eq:tat_X2}
 X_2=\sum_{i=0}^K \frac{(\nu_i-t\pi_i)^2}{t\pi_i},
\end{equation}
``which, under the randomness hypothesis, has an approximate \emph{$\chi^2$-distribution with $K$ degrees of freedom}''. Here the theoretical values $\pi_i$ can be replaced by approximating numerical values which for certain pairs $M,K$ are provided in Section 3.4 of \cite{nist}. (We remark that for fixed $M$, the choice of $K$ and the computation of the values approximating $\pi_i$ is based on the analysis of distribution of the longest run in random walks; see, e.g. \cite[Chapter~7]{Revesz}.)

We will show that the statistic $X_2$ in \eqref{eq:tat_X2} can be estimated in the following way:

\begin{thm}\label{thm:longest_run}
We have
\[
 X_2\leq \frac{M}{N} \left(\sum_{r=1}^M \binom{M}{r} Q_r(E_N)\right)^2.
\]
\end{thm}

Note that this estimate gives a good bound for $X_2$ only if $N$ is large in terms of $M$; the first table in \cite{nist}, pp. 2-8
seems to indicate that this can be assumed.

\begin{proof}
 We will use the following notations:
 
 For $Z\in\mathbb{N}$, let $\Phi_Z$ be the set of the binary sequences
 \[
  F_Z=(f_1,f_2,\dots, f_Z)\in\{-1,+1\}^Z,
 \]
for such a sequence $F_Z$ let $\psi(\F_Z)$ denote the length of the longest run of $+1$'s in $F_Z$, and for $j\in\mathbb{N}$, $jM\leq Z$, write
$F_Z^{(j,M)}=(f_{(j-1)M+1},f_{(j-1)M+2},\dots, f_{jM})$.

Then for $i=0,1,\dots, K$, by the definition of $\nu_i$ and \eqref{eq:NMT} we have
\begin{equation}\label{eq:nu}
 \nu_i=\sum_{\substack{1\leq j\leq t \\ \psi(E_N^{(j,M)})\in\mathcal{P}_i}}1,
\end{equation}
and the expectation of $\nu_i$ choosing any $F_N\in\Phi_N$ with equal probability $1/2^N$ is
\begin{equation}\label{eq:expectation}
 \mathbb{E} (\nu_i)= \mathbb{E}\left(\sum_{\substack{1\leq j\leq t \\ \psi(E_N^{(j,M)})\in\mathcal{P}_i}}1\right)=
 t \cdot \mathbb{E}\left(\sum_{\substack{G\in\Phi_M \\ \psi(G)\in\mathcal{P}_i}}1\right)=t\pi_i.
\end{equation}
Let $G^{(1)},G^{(2)},\dots, G^{(\gamma_i)}$ be the sets $G$ counted in the last sum, and write $\mathcal{G}_i=\{G^{(1)},G^{(2)},\dots, G^{(\gamma_i)}\}$ and $G^{(j)}=(g_1^{(j)},g_2^{(j)},\dots, g_M^{(j)})$. Each of these sets $G^{(j)}$ contributes by 1 to this sum, and they are to be selected with probability $1/2^M$ uniformly. Thus it follows from  \eqref{eq:expectation} that
\begin{equation}\label{eq:exp}
 \mathbb{E}(\nu_i)=t \pi_i=t \frac{|\mathcal{G}_i|}{2^M}=\frac{t}{2^M}\gamma_i.
\end{equation}
Moreover, it follows from \eqref{eq:partition} that each of the $2^M$ sets $G\in\Phi_M$ is counted in exactly one $\mathcal{G}_i$ with weight 1, thus we have
\begin{equation}\label{eq:G}
 \sum_{i=0}^K|\mathcal{G}_i|=\sum_{i=0}^K\gamma_i=2^M.
\end{equation}
Now we will estimate $\nu_i$ for $0\leq i\leq K$. Consider a subset $G^{(\ell)}=(g_1^{(\ell)},g_2^{(\ell)},\dots,g_M^{(\ell)})\in\mathcal{G}_i$. Then for $j=1,2,\dots,t$ clearly we have
\[
 \prod_{x=1}^M\frac{1+e_{(j-1)M+x}g_x^{(\ell)}}{2}=\left\{
 \begin{array}{cl}
  1 & \text{if } E_{N}^{(j,M)}=G^{(\ell)},\\
  0 & \text{if } E_{N}^{(j,M)} \neq G^{(\ell)},
 \end{array}
 \right.
\]
whence
\[
 \sum_{\ell=1}^{\gamma_i}\prod_{x=1}^M\frac{1+e_{(j-1)M+x}g_x^{(\ell)}}{2}=\left\{
 \begin{array}{cl}
  1 & \text{if } E_{N}^{(j,M)} \in \{G^{(1)},G^{(2)},\dots, G^{(\gamma_i)}\}=\mathcal{G}_i,\\
  0 & \text{if } E_{N}^{(j,M)} \not \in \mathcal{G}_i,
 \end{array}
 \right.
\]
so that by \eqref{eq:nu} we have
\begin{align}\label{eq:expansion}
  \nu_i&=\sum_{\substack{1\leq j\leq t \\ \psi(F_N^{(j,M)})\in\mathcal{P}_i}}1=\sum_{1\leq j\leq t}\sum_{E_N^{(j,M)}\in\mathcal{G}_i}1=  \sum_{j=1}^t\sum_{\ell=1}^{\gamma_i}\prod_{x=1}^M\frac{1+e_{(j-1)M+x}g_x^{(\ell)}}{2}\notag \\
  &= \sum_{\ell=1}^{\gamma_i} \sum_{j=1}^t \left(\frac{1}{2^M}+\frac{1}{2^M}\sum_{r=1}^M \sum_{1\leq x_1<\dots< x_r\leq M} g_{x_1}^{(\ell)}\dots g_{x_r}^{(\ell)} e_{(j-1)M+x_1}\dots e_{(j-1)M+x_r} \right)\notag \\
  &= \frac{t}{2^M}\gamma_i + \frac{1}{2^M}\sum_{\ell=1}^{\gamma_i} \left(\sum_{r=1}^M \sum_{1\leq x_1<\dots< x_r\leq M} g_{x_1}^{(\ell)}\dots g_{x_r}^{(\ell)} \sum_{j=1}^t e_{(j-1)M+x_1}\dots e_{(j-1)M+x_r} \right).
\end{align}
It follows from \eqref{eq:exp} and \eqref{eq:expansion} that
\begin{align}\label{eq:deviation}
 |\nu_i-t\pi_i|&= \frac{1}{2^M}\left|\sum_{\ell=1}^{\gamma_i} \left(\sum_{r=1}^M \sum_{1\leq x_1<\dots< x_r\leq M} g_{x_1}^{(\ell)}\dots g_{x_r}^{(\ell)} \sum_{j=1}^t e_{(j-1)M+x_1}\dots e_{(j-1)M+x_r} \right) \right|\notag \\
 &\leq  \frac{1}{2^M}\sum_{\ell=1}^{\gamma_i} \sum_{r=1}^M \sum_{1\leq x_1<\dots< x_r\leq M} \left|g_{x_1}^{(\ell)}\dots g_{x_r}^{(\ell)}\right| \left| \sum_{j=1}^t e_{(j-1)M+x_1}\dots e_{(j-1)M+x_r} \right| \notag\\
  &=\frac{\gamma_i}{2^M}\sum_{r=1}^M \binom{M}{r}Q_r(E_N)=\pi_i \sum_{r=1}^M \binom{M}{r}Q_r(E_N).
\end{align}
By \eqref{eq:NMT}, \eqref{eq:exp}, \eqref{eq:G} and \eqref{eq:deviation} we have
\begin{align*}
 X_2 
 =
 \sum_{i=0}^K \frac{(\nu_i-t\pi_i)^2}{t\pi_i}
 \leq &\sum_{i=0}^K
 \frac{\pi_i}{t} \left(\sum_{r=1}^M \binom{M}{r}Q_r(E_N)\right)^2 \\
 &
 = \frac{1}{t}\left(\sum_{r=1}^M \binom{M}{r}Q_r(E_N)\right)^2 \sum_{i=0}^K \pi_i
 =\frac{M}{N}\left(\sum_{r=1}^M \binom{M}{r}Q_r(E_N)\right)^2 
\end{align*}
which completes the proof of the theorem.
\end{proof}
We remark that in Theorem~\ref{thm:longest_run} the statistic $X_2$ is
estimated in terms of the combined pseudorandom measure $Q_k$, while in the
most important constructions studied in Theorems~\ref{thm:GMS2004},
\ref{thm:inverse-W}, \ref{thm:inverse-Liu} 
and \ref{thm:ec} only the measures $W$ and $C_k$ were estimated, and no
estimates are known for the measures $Q_k$ (and the situation is similar in
most of the other constructions). However, this gap can be bridged easily,
since in most cases the estimate of $Q_k$ can be reduced easily to the
estimate of $C_k$. 
For example, in case of the Legendre symbol 
construction \eqref{constr:legendre} studied in Theorem \ref{thm:GMS2004}, 
we can show that if the 
sequence $E_p$ is defined by \eqref{constr:legendre} in Theorem
\ref{thm:GMS2004} and we assume that all the assumptions in the theorem hold, 
then we have
\[
 Q_k(E_p)\leq C_k(E_p)+2k,
\]
and in case of the two other constructions similar results could be proved. 

\section{The ``linear complexity test''}\label{sec:linCompl}

The \textit{linear complexity} $L(\tilde{E}_N)$ of a \emph{bit} sequence $\tilde{E}_N=(\tilde{e}_1,\dots, \tilde{e}_N)\in\{0,1\}^N$  is defined as the length $L$ of a shortest linear recurrence relation (linear feedback shift register -- LFSR)
\begin{equation*}
 \tilde{e}_{n+L} \equiv c_{L-1}\tilde{e}_{n+L-1}+\dots +c_1\tilde{e}_{n+1}+c_0\tilde{e}_n \pmod 2, \quad 1\leq n\leq N-L
\end{equation*}
where $c_0,\dots, c_{L-1}\in \{0,1\}$, that $\tilde{E}_N$  satisfies,  
with the convention that $L(\tilde{E}_N)=0$ if $\tilde{E}_N=(0,\dots, 0)$, and $L(\tilde{E}_N)=N$ if  $\tilde{E}_N=(0,\dots, 0,1)$. For binary sequence $E_N$ of form \eqref{eq:E} we also define the linear complexity as $L(E_N)=L(\tilde{E}_N)$ with $\tilde{e}_n=(1+e_n)/2$.

The linear complexity is a measure for the unpredictability of a sequence. A large linear complexity is necessary (but not sufficient) for cryptographic applications. The \emph{linear complexity test} appears as Section 2.10 in \cite{nist}. We quote: ``The purpose of this test is to determine whether or not the sequence is complex enough to be considered random. Random sequences are characterized by longer LFSRs. An LFSR that is too short implies non-randomness.''

Brandst\"atter and Winterhof \cite{BW2006} 
showed that a small correlation measure implies large linear complexity:
\begin{equation}\label{eq:BrandstatterWinterhof}
 L(E_N)\geq N- \max_{1\leq k\leq L(E_N)+1}C_k(E_N).
\end{equation}
This result provides a lower bound for the linear complexity of sequences generated by using the Legendre symbol \eqref{constr:legendre} and elliptic curves \eqref{constr:ec}. Namely, if $E_p$ is a sequence generated by \eqref{constr:legendre} using a squarefree polynomial $f(x)\in\F_p[x]$ of degree $k$ and 2 is a primitive root modulo $p$, then Theorem \ref{thm:GMS2004} and \eqref{eq:BrandstatterWinterhof} imply that
\[
 p\leq L(E_p)+ \max_{1\leq k\leq L(E_p)+1}C_k(E_p)\ll k L p^{1/2}\log p,
\]
so that
\begin{equation}\label{eq:linCompl-Legendre}
 L(E_p)\gg \frac{p^{1/2}}{k \log p}.
\end{equation}

Similarly, if the sequence $E_T$ is generated by \eqref{constr:ec} using a squarefree function $f(x,y)\in\F_p[\E]$ with degree $k$ and 2 is a primitive root modulo $T$, then Theorem \ref{thm:ec} and \eqref{eq:BrandstatterWinterhof} imply that
\[
  L(E_T)\gg \frac{p^{1/2}}{k \log T}.
\]
By the celebrated Hasse-Weil Theorem (see e.g. \cite{washington}, Theorem 4.2) 
we have $\big|p+1-|\E(\F_p)|\big|\leq 2p^{1/2}$, thus
\begin{equation}\label{eq:linCompl-ec}
  L(E_T)\gg \frac{T^{1/2}}{k \log T}.
\end{equation}

In practice, the bounds \eqref{eq:linCompl-Legendre} and \eqref{eq:linCompl-ec} are sometimes sufficient. However, the linear complexity of a truly random binary sequence of length $N$ is around $N/2$, thus in more demanding applications one may have to show that the linear complexity of the given sequence is near $N/2$ (or at least it is $\gg N$); the linear complexity of the sequence can be determined by using the well-known Berlekamp-Massey algorithm \cite{BerlekampMassey}.

In even more demanding cases one may need an even more precise study of the complexity properties of the sequence. In \cite{nist} this is done by the ``linear complexity test'' described in \cite[p. 2--24]{nist}. This test requires a more controlled distribution of the linear complexity of the sequences. Namely, it compares the linear complexity within blocks of length $M$ to the expected value of the linear complexity
\[
 \mu _M=\frac{M}{2} + \frac{4+ r_2(M)}{18}
\]
(here again $r_2(M)$ is the non-negative remainder of $M$ modulo 2).

We quote: ``Partition the $N$-bit sequence into $t$ independent blocks of $M$ bits, where $N=t\cdot M$.'' Then ``determine the linear complexity $L_i$ of each of the $t$ blocks ($i=1,\dots, t$)''. ``For each substring, calculate a value of $T_i$ where
\[
T_i=(-1)^M\cdot (L_i-\mu_M)+\frac{2}{9}.''
\]
Define the intervals: 
\[
 \begin{array}{ll}
  I_0=(-\infty, -2,5], & \\
  I_j=(-2.5+j-1,-2.5+j], & j=1,\dots, 5,\\
  I_6=(2.5,\infty), &
 \end{array}
\]
and put $v_j=|\{i: T_i\in I_j, \ i=1,\dots, t\}|$. 

Finally, we define the statistic
\[
 X_3=\sum_{j=0}^6 \frac{(v_j-t\cdot \pi_j)^2}{t\cdot \pi_j},
\]
where $\pi_j$ ($i=0,\dots, 6$) are the probabilities for the classes $I_j$:
\[
 \pi_j=P\left((-1)^M\cdot (L(E_M)-\mu_M)+\frac{2}{9}\in I_j\right),
\]
where $E_M$ is chosen uniformly from $\{-1,+1\}^M$. The acceptance of the sequence depends on the value of the statistic $X_3$: one has to compute the 
``$P$-value'' defined in \cite{nist}, p. 2-25, (7) and if the 
``$P$-value'' is $\ge 0.01$, then the sequence passes the test. 

If the sequence $E_N$ to be tested possesses strong pseudorandom properties in terms of the measures described in Section~\ref{sec:measures}, and the length $M$ of the blocks is much smaller than the length $N$ of the sequence (say, we have $M=o(\log N)$), then one could give a reasonable upper bound for the statistic $X_3$ by the method used in Section~\ref{sec:longest_run} (although here even more work and computation would be needed). However, according to ``input size recommendation'' in \cite[Section~2.10.7]{nist}, $M$ must be very large ($500\leq M\leq 5000$) so that to have $M=o(\log N)$, $N$ must be huge (say, $N>10^{10000}$), thus we will not present the details here. This, of course, does not mean that shorter sequences with good pseudorandom properties fail this test and, indeed, the numerical examples in Section~\ref{sec:numerical} will show that sequences of this type tend to pass this test, but we cannot show that this is necessarily so.
\begin{remark}
As mentioned previously, small linear complexity implies non-randomness. 
However, recent results show
that there are many sequences whose 
linear complexity is very
near 
to its expected value but which also have some
cryptographic weakness: Winterhof and
the first author provided a large class of highly predictable sequences whose
linear complexity is close to its mean \cite{MW2016+}. 
A simple way to eliminate such sequences is to consider also 
the \emph{expansion complexity} of the sequences defined 
in \cite{D2012,MNW2016+}.
\end{remark}

\section{Discrete Fourier Transform (Spectral) Test}\label{sec:spectral}

The ``NIST tests'' writes: ``The purpose of this test
is to detect periodic features (i.e., repetitive patterns that are
near each other) in the tested sequence that would indicate a
deviation from the assumption of randomness''.

This is one of the tests which are too complicated to estimate the
statistic to be studied by using our measures of
pseudorandomness. Instead, we will do the following: we show that
the goal of the test described above can be also achieved by using our
measures of pseudorandomness and, indeed, the combined pseudorandom
measure of order $k$ described above is especially suitable for
this. Take the following example:

\begin{example}
  Consider the $4$-tuple $+1$, $-1$, $-1$, $+1$, and repeat it
  $M = 500000$ times. Then letting $N = 4M = 2000000$, we
  get a binary sequence $E_N = (e_1,e_2,\ldots,e_N)$ with 
  $e_{4k-3} = +1$, $e_{4k-2} = -1$, $e_{4k-1} = -1$, $e_{4k} = +1$
  for $k=1,2,\ldots, M$.
  This sequence is periodic with period $4$. 
  Its combined pseudorandom measure of order $4$ can be estimated in
  the following way:
  \begin{displaymath}
    Q_4(E_N)
    \geq
    \left| \sum_{j=0}^{M-1} e_{4j+1}e_{4j+2}e_{4j+3}e_{4j+4} \right|
    =
    \left| \sum_{j=0}^{M-1} 1 \right|
    =
    M 
    =
    \frac{N}{4},
  \end{displaymath}
  so that this measure is big, is as large as $\frac14$ times the
  length of the sequence, much larger than the optimal 
  $\asymp N^{1/2}$.
  This fact is reflected in the periodicity, thus the sequence is far from
  being of pseudorandom nature.

  Of course, if a sequence is not completely periodic but is only
  almost periodic with period $k$, than its $Q_k$ measure is still
  large.

  Applying the Discrete Fourier Transform Test for testing the
  sequence $E_N$ defined above with $20$ samples of length $100000$, 
  we find that it fails this test (strongly).

  So that both approaches point out the periodic nature of $E_N$, thus
  \textit{it fails both tests}.
\end{example}
Now let us study a more complicated example.
\begin{example}
  Consider two especially important special sequences: the
  Rudin-Shapiro sequence 
  (defined by 
  \begin{math}
    (-1)^{\sum_{i} \varepsilon_i(n)\varepsilon_{i+1}(n)}
  \end{math}
  where $\varepsilon_i$ denotes the $i$-th binary digit of $n$
  )
  and the Thue-Morse sequence
  (defined by 
  \begin{math}
    (-1)^{\sum_{i} \varepsilon_i(n)}
  \end{math}
  ).
  It is known \cite{MS1998} that in both cases if we take a sequence of
  length $N$ then its correlation measure of order $2$ is very large, it is
  $\gg N$ which is again much larger than the optimal $\asymp N^{1/2}$
  so that in terms of the measures of pseudo randomness described
  above they are both far from being of pseudorandom nature.

  But what about the Discrete Fourier Transform Test, do these
  sequences also fail this test?
  First consider the Rudin-Shapiro sequence:
\begin{center}
\begin{figure}[H]
{\small
\begin{Verbatim}
------------------------------------------------------------------------------
RESULTS FOR THE UNIFORMITY OF P-VALUES AND THE PROPORTION OF PASSING SEQUENCES
------------------------------------------------------------------------------
 C1  C2  C3  C4  C5  C6  C7  C8  C9 C10  P-VALUE  PROPORTION  STATISTICAL TEST
------------------------------------------------------------------------------
  3   3   0   2   3   3   1   2   3   0  0.637119     20/20      Frequency
  4   0   0   0   3   2   0   0   0  11  0.000000 *   16/20   *  BlockFrequency
  2   3   5   2   1   2   1   1   1   2  0.637119     20/20      CumulativeSums
  2   1   6   2   0   3   2   1   1   2  0.213309     20/20      CumulativeSums
  0   0   0   0   0   0   0   0   0  20  0.000000 *   20/20      Runs
 20   0   0   0   0   0   0   0   0   0  0.000000 *    0/20   *  LongestRun
 20   0   0   0   0   0   0   0   0   0  0.000000 *    0/20   *  Rank
 20   0   0   0   0   0   0   0   0   0  0.000000 *    0/20   *  FFT
 20   0   0   0   0   0   0   0   0   0  0.000000 *    0/20   *  Universal
 20   0   0   0   0   0   0   0   0   0  0.000000 *    0/20   *  ApproximateEntropy
 20   0   0   0   0   0   0   0   0   0  0.000000 *    0/20   *  Serial
 20   0   0   0   0   0   0   0   0   0  0.000000 *    0/20   *  Serial
 20   0   0   0   0   0   0   0   0   0  0.000000 *    0/20   *  LinearComplexity
\end{Verbatim}
}
\caption{Results of 13 NIST tests for the Rudin-Shapiro sequence}\label{fig:RS}
\end{figure}
\end{center}
(Here and in some further tables an asterisk indicates after the 
column ``$P$-value'' that the $P$-values belonging to the sequence studied 
and the test named in the last column are non-uniform, while after 
column ``proportion'' it denotes that many or all of these sequences 
fail the test in question.) Now consider the Thue-Morse sequence:
\begin{center}
\begin{figure}[H]
{\small
\begin{Verbatim}
------------------------------------------------------------------------------
RESULTS FOR THE UNIFORMITY OF P-VALUES AND THE PROPORTION OF PASSING SEQUENCES
------------------------------------------------------------------------------
 C1  C2  C3  C4  C5  C6  C7  C8  C9 C10  P-VALUE  PROPORTION  STATISTICAL TEST
------------------------------------------------------------------------------
  0   0   0   0   0   0   0   0   0  20  0.000000 *   20/20      Frequency
  0   0   0   0   0   0   0   0   0  20  0.000000 *   20/20      BlockFrequency
  0   0   0   0   0   0   0   0   0  20  0.000000 *   20/20      CumulativeSums
  0   0   0   0   0   0   0   0   0  20  0.000000 *   20/20      CumulativeSums
 20   0   0   0   0   0   0   0   0   0  0.000000 *    0/20   *  Runs
 20   0   0   0   0   0   0   0   0   0  0.000000 *    0/20   *  LongestRun
 20   0   0   0   0   0   0   0   0   0  0.000000 *    0/20   *  Rank
 19   0   0   0   0   0   0   0   1   0  0.000000 *    2/20   *  FFT
 20   0   0   0   0   0   0   0   0   0  0.000000 *    0/20   *  Universal
 20   0   0   0   0   0   0   0   0   0  0.000000 *    0/20   *  ApproximateEntropy
 20   0   0   0   0   0   0   0   0   0  0.000000 *    0/20   *  Serial
 20   0   0   0   0   0   0   0   0   0  0.000000 *    0/20   *  Serial
 20   0   0   0   0   0   0   0   0   0  0.000000 *    0/20   *  LinearComplexity
\end{Verbatim}
}
\caption{Results of 13 NIST tests for the Thue-Morse sequence}\label{fig:TM}
\end{figure}
\end{center}
So that both sequences fail this test. The Fourier transform of the 
Rudin-Shapiro polynomial, whose maximum modulus is very close to its 
$L^2$ norm, and this is a very unusual property (see \cite{Sh}). Similarly, 
the Fourier transform of the Thue-Morse sequence has a very small 
$L^1$ norm (see \cite{FM}). This explains why they fail the Discrete Fourier 
Transform test which detects that they are far from the 
DFT of a random sequence. 
\end{example}
Our examples show that both approaches can be used effectively for
detecting some sort of periodicity.

\section{Numerical calculation}\label{sec:numerical}

In the previous sections we showed that those binary sequences $E_N\in\{-1,+1\}^N$ whose pseudorandom measures $W(E_N)$ and $C_k(E_N)$ are small, also have  strong pseudorandom properties in terms of the NIST tests \emph{a priori}, i.e. they provably  pass or ``almost pass'' most of the NIST tests. In this section we test sequences constructed by principles described in Section \ref{sec:construction} \emph{a posteriori}. The examples show that these sequences typically pass the NIST tests, even if we can only prove a slightly weaker pseudorandomness.

We use ``Statistical Test Suite for random and pseudorandom number generators for cryptographic application'' (\texttt{sts-1.4}) from the National Institute of Standards and Technology (NIST). We chose the following parameters for the test suite.
\begin{figure}[!ht]
\begin{center}
 \begin{tabular}{ll}
\texttt{BlockFrequency} & $M=128$\\
\texttt{OverlappingTemplate} & $m=9$\\
\texttt{ApproximateEntropy} & $m=10$ \\
\texttt{LinearComplexity} & $M=500$\\
\end{tabular}
\caption{Parameter choices for NIST test suite}
\end{center}
\end{figure}

In order to save space we omit the results of the non-overlapping template matching (\texttt{NonOver\-lappingTemplate}), the random excursions (\texttt{RandomExcursions}) and the random excursions variant tests (\texttt{RandomExcursionsVariant}).

The results of the tests are given in Figures \ref{fig:legendre}, \ref{fig:inverse} and \ref{fig:ec}. Columns \texttt{C1} up to \texttt{C10} correspond to the frequency specific to the test. Then the \texttt{P-VALUE} is the result of the application of a $\chi^2$-test, and \texttt{PROPORTION} is the proportion of sequences that pass the test.

\subsection{Sequences generated using the Legendre symbol}

We constructed 20 sequences with length $p=10^5+3$ by \eqref{constr:legendre} with the first 20 squarefree polynomial of degree 31 with respect to the lexicographic ordering: $f_i(x)=x^{31}+i$ ($i=1,\dots, 20$). Since 2 is a primitive root modulo $p=10^5+3$, Theorem \ref{thm:GMS2004} implies that all the sequences $E_p(i)$ generated with $f_i(x)$ ($i=1,\dots, 20$) have small well-distribution and correlation measures. 

\begin{figure}[H]
{\small
\begin{Verbatim}
------------------------------------------------------------------------------
RESULTS FOR THE UNIFORMITY OF P-VALUES AND THE PROPORTION OF PASSING SEQUENCES
------------------------------------------------------------------------------
 C1  C2  C3  C4  C5  C6  C7  C8  C9 C10  P-VALUE  PROPORTION  STATISTICAL TEST
------------------------------------------------------------------------------
  0   0   0   0   0   0   0   0   0  20  0.000000 *   20/20      Frequency
  1   4   2   2   2   2   3   3   1   0  0.739918     20/20      BlockFrequency
  0   1   0   0   0   1   1   4   7   6  0.000199     20/20      CumulativeSums
  0   1   0   0   0   1   1   4   7   6  0.000199     20/20      CumulativeSums
  4   4   1   1   3   1   1   3   0   2  0.437274     20/20      Runs
  1   2   2   1   5   1   2   3   1   2  0.637119     20/20      LongestRun
  2   1   0   3   1   0   4   2   5   2  0.213309     20/20      Rank
  0   0   0   0   0  10   0  10   0   0  0.000000 *   20/20      FFT
  2   6   0   0   1   6   2   1   2   0  0.006196     19/20      OverlappingTemplate
  0  20   0   0   0   0   0   0   0   0  0.000000 *   20/20      Universal
  1   1   6   1   1   2   2   1   4   1  0.162606     20/20      ApproximateEntropy
  1   2   1   1   3   2   2   5   1   2  0.637119     20/20      Serial
  0   4   1   3   1   5   0   2   0   4  0.066882     20/20      Serial
  1   4   1   2   2   3   0   1   3   3  0.637119     20/20      LinearComplexity
- - - - - - - - - - - - - - - - - - - - - - - - - - - - - - - - - - - - - - - - -

\end{Verbatim}
}
\caption{Results of 14 NIST tests for sequences generated by using the Legendre symbol}\label{fig:legendre}
\end{figure}

\subsection{Sequences generated using the multiplicative inverse}

We took $p=2\cdot 10^5+3$ and considered the polynomials
\[
f_i(x)=x\cdot  \prod_{j=15(i-1)+1}^{15i} (x^2+j^2),\quad i=1,\dots, 20.
\]
Since $2\cdot 10^5+3\equiv 3 \pmod 4$, $-1$ is quadratic non-residue, and the least non-negative remainders of $-j^2$ ($j=1,\dots,300$) modulo $p$, $r_{p}(-j^2)$,  are also quadratic non-residues. Then these polynomials satisfy the conditions of Theorem \ref{thm:inverse-Liu}, thus they have small well-distribution and correlation measures. However the sequences generated by \eqref{constr:inverse} with the polynomials  $f_i(x)$ have a non-trivial symmetry. Namely, $f_i(-x)=-f_i(x)$, so $e_{n}=-e_{p-n}$ for all $1\leq n< p$ if the sequence $E_{p}=(e_0,\dots, e_{p-1})\in\{-1,+1\}^p$ is generated such a way. (For tools to detect such symmetries see \cite{gySymmetry}). To avoid this phenomenon we just considered the first half of the sequences: $E(i)_{(p+1)/2}=\{e_0(i),\dots, e_{(p-1)/2}(i)\}$, where $e_n(i)$ ($0\leq n<p/2$) is defined by the rule \eqref{constr:inverse}. In this way we obtained 20 sequences of length 100002.
\begin{figure}[H]
{\small
\begin{Verbatim}
------------------------------------------------------------------------------
RESULTS FOR THE UNIFORMITY OF P-VALUES AND THE PROPORTION OF PASSING SEQUENCES
------------------------------------------------------------------------------
 C1  C2  C3  C4  C5  C6  C7  C8  C9 C10  P-VALUE  PROPORTION  STATISTICAL TEST
------------------------------------------------------------------------------
  3   2   1   3   0   2   3   4   1   1  0.637119     20/20      Frequency
  5   2   0   1   3   4   2   1   1   1  0.275709     19/20      BlockFrequency
  3   1   0   2   2   1   3   0   3   5  0.275709     20/20      CumulativeSums
  3   2   2   1   4   0   3   1   4   0  0.350485     20/20      CumulativeSums
  2   3   2   5   3   1   1   2   0   1  0.437274     20/20      Runs
  2   1   1   2   5   2   3   0   2   2  0.534146     20/20      LongestRun
  2   3   3   3   0   0   2   1   4   2  0.534146     20/20      Rank
  2   2   1   5   3   3   1   0   2   1  0.437274     20/20      FFT
  4   1   1   1   3   0   2   4   1   3  0.437274     20/20      OverlappingTemplate
  0   0   0   0   0   0  20   0   0   0  0.000000 *   20/20      Universal
  5   4   1   2   2   2   1   1   2   0  0.350485     20/20      ApproximateEntropy
  1   2   1   2   1   0   2   7   4   0  0.017912     20/20      Serial
  2   1   2   0   3   3   3   2   0   4  0.534146     19/20      Serial
  1   3   3   0   3   2   1   4   1   2  0.637119     20/20      LinearComplexity
- - - - - - - - - - - - - - - - - - - - - - - - - - - - - - - - - - - - - - - - -

\end{Verbatim}
}
\caption{Results of 14 NIST tests for sequences generated by using the multiplicative inverse}\label{fig:inverse}
\end{figure}

\subsection{Sequences generated using elliptic curves}

In order to generate sequences with elliptic curves we chose pseudorandom curves and points following the NIST recommendation (FIPS 186-3). We took the prime $p=10^5+3$ and a pseudorandom elliptic curve of the form
\[
 y^2=x^3-3x+b 
\]
over $\F_p$ with the additional restriction, that the number $T$ of the $\F_p$-rational points is prime and 2 is a primitive root modulo $T$. 
Then we selected a pseudorandom point $P$ on the curve. 

Our parameters were the following: 
\[
 \E:  y^2=x^3-3x+74439 \quad \text{over } \F_{10^5+3}.
\]
Its cardinality $T$ is $100523$. The point was $P=(85611,76395)$. We took the functions $f_i(x,y)=x^{31}+x+y+i$ ($i=0,\dots, 19$). Since 2 is a primitive root modulo $T$, and the functions $f_i(x,y)$ ($i=0,\dots, 19$) are not perfect squares, Theorem \ref{thm:ec} implies, that the well-distribution and correlation measures of sequences generated by the polynomials $f_i(x,y)$ ($i=0,\dots, 19$) are small.

\begin{figure}[H]
\centering
{\small
\begin{Verbatim}
------------------------------------------------------------------------------
RESULTS FOR THE UNIFORMITY OF P-VALUES AND THE PROPORTION OF PASSING SEQUENCES
------------------------------------------------------------------------------
 C1  C2  C3  C4  C5  C6  C7  C8  C9 C10  P-VALUE  PROPORTION  STATISTICAL TEST
------------------------------------------------------------------------------
  4   1   3   0   1   2   2   2   2   3  0.739918     20/20      Frequency
  0   3   4   0   3   2   2   1   2   3  0.534146     20/20      BlockFrequency
  1   5   2   1   3   1   0   2   3   2  0.437274     20/20      CumulativeSums
  4   0   2   4   2   1   3   1   2   1  0.534146     20/20      CumulativeSums
  4   2   4   3   2   0   2   1   1   1  0.534146     20/20      Runs
  5   3   1   3   4   0   0   3   1   0  0.090936     20/20      LongestRun
  3   4   1   0   2   1   1   1   4   3  0.437274     19/20      Rank
  3   3   2   2   4   1   1   1   1   2  0.834308     20/20      FFT
  4   1   1   2   1   1   6   2   0   2  0.122325     20/20      OverlappingTemplate
  0   0   0   0   0   0   0   0   0  20  0.000000 *   20/20      Universal
  5   1   2   1   0   5   3   1   1   1  0.122325     20/20      ApproximateEntropy
  1   2   2   3   4   0   1   3   2   2  0.739918     19/20      Serial
  1   1   4   2   1   2   4   1   2   2  0.739918     19/20      Serial
  1   3   3   6   0   2   2   1   1   1  0.162606     20/20      LinearComplexity
- - - - - - - - - - - - - - - - - - - - - - - - - - - - - - - - - - - - - - - - -

\end{Verbatim}
}
\caption{Results of 14 NIST tests for sequences generated by using the elliptic curves}\label{fig:ec}
\end{figure}

Summarizing: we have considered altogether 60 binary sequences which have
been proved to possess good pseudorandom properties in terms of 
the pseudorandom measures described in Section \ref{sec:measures}, 
and we tested them by 14 NIST tests. 834 times out of 840 the sequences 
passed the test so that the NIST tests confirmed the good pseudorandom 
quality of the sequence. 

\section{Conclusion Remarks}

The tables in Section 8 show that the sequences ``good'' in terms of the 
measures defined in Section 2 are usually also ``good'' in terms of the 
NIST tests. Does this mean that we may eliminate the 
NIST tests, replace them by estimating the pseudorandom measures 
described above? Certainly not: both methods have advantages and 
disadvantages. The greatest advantage of using the measures of pseudorandomness
described above is that at least for certain special sequences they
enable us to provide ``a priori'', ``theoretical'' testing without any
further computations. An other advantage of this method is that in 
many constructions (like the ones described in Section 3) we can 
give a good upper bound (simultaneously by just a single computation)
for the correlation measure of order $k$ of a sequence of length $N$ 
for every $k$ with 
$k<N^c$ (say, with $c=1/4$), and since it is known \cite{CMS2002}
that the correlation measures of order $k$, resp. $\ell$ are independent if 
$k\nmid\ell$ and $\ell\nmid k$, thus by estimating the correlation measures 
whose order is less than $N^c$, we test $N^{c'}$ (with $c'<c$)
independent pseudorandom properties of the sequence, while in the 
NIST tests only 15 properties are tested. On the other hand, 
the disadvantage of this
approach is that in most cases it is very difficult to estimate
these measures, e.g. there are no algorithms for estimating 
correlation measure of high order. On the other hand, NIST provides good
and fast algorithms for performing these tests, while its
disadvantage is that it can not be used for ``a priori'',
``theoretical'' testing.

In Section 2 we described only the most important measures of 
pseudorandomness and in Section 3 we presented only three 
constructions for sequences having good pseudorandom properties 
in terms of these measures. There are also other measures of 
pseudorandomness and many further constructions; a survey of these 
measures and constructions; a survey of these measures and constructions 
is presented in \cite{Gy2013}. In this paper we have been focusing on 
studying pseudorandom properties of 
\textit{single binary sequences}. However, 
as we mentioned in Section 1, if our goal is to test the quality 
of a pseudorandom generator, then it is not enough to restrict ourselves to 
testing single sequences; one also has to continue the work by using the 
tools of cryptanalysis for testing the \textit{family  of the sequences} 
generated by the given algorithm. Tools for helping this work also have been 
introduced (in the spirit of the measures described in Section 2):
family complexity \cite{AKMS}, cross-correlation measure \cite{GMS}, 
distance minimum and avalanche effect \cite{To}, etc., and 
in each of these cases constructions have been presented for families 
possessing good pseudorandom properties in terms of these measures. A
survey of this type of papers is presented in \cite{SA}.

\end{document}